\documentclass[a4paper,12pt,leqno]{amsart}
\usepackage[T1]{fontenc}
\usepackage{enumerate}
\usepackage{amsmath,amssymb,amsthm,amsmath}

\usepackage{units}
\usepackage{mathrsfs,mathtools,graphicx,dsfont}
\usepackage{pstricks,pst-node,pst-plot}
\usepackage{etoolbox}

\newtheorem{theorem}{Theorem}[section]
\newtheorem{lemma}[theorem]{Lemma}
\newtheorem{proposition}[theorem]{Proposition}

\newtheorem{corollary}[theorem]{Corollary}

\theoremstyle{definition}
\newtheorem{definition}[theorem]{Definition}
\newtheorem{example}[theorem]{Example}

\theoremstyle{remark}
\newtheorem{remark}[theorem]{Remark}

\numberwithin{equation}{section}

\def\R{{\mathbb R}}

\def\N{{\mathbb N}}

\def\D{{\mathbb D}} 

\def\K{{\mathbb K}}

\def\C{{\mathbb C}}
\def\cD{{\mathcal D}}

\def\cL{{\mathbb L}}

\def\cD{{\mathcal D}}

\def\cH{{\mathcal H}}

\def\cL{{\mathcal L}}

\def\tr{\operatorname{tr}}
\def\Im{\operatorname{Im}}
\def\Re{\operatorname{Re}}

\def\div{\operatorname{div}}

\def\loc{\operatorname{loc}}
\def\diag{\operatorname{diag}}

\title[From forms to semigroups]{From forms to semigroups}

\author{Wolfgang Arendt }
\address{Institute of Applied Analysis,  University of Ulm, D - 89069 Ulm, 
Germany}
\email{wolfgang.arendt@uni-ulm.de}

\author{A. F. M. ter Elst} 
\address{Department of Mathematics, University of Auckland, Private 
Bag 92019, Auckland 1142, New Zealand}
\email{terelst@math.auckland.ac.nz}

\newcommand{\one}{1\hspace{-4.5pt}1}

\makeatletter
\def\eqnarray{\stepcounter{equation}\let\@currentlabel=\theequation
\global\@eqnswtrue
\tabskip\@centering\let\\=\@eqncr
$$\halign to \displaywidth\bgroup\hfil\global\@eqcnt\z@
  $\displaystyle\tabskip\z@{##}$&\global\@eqcnt\@ne
  \hfil$\displaystyle{{}##{}}$\hfil
  &\global\@eqcnt\tw@ $\displaystyle{##}$\hfil
  \tabskip\@centering&\llap{##}\tabskip\z@\cr}

\def\endeqnarray{\@@eqncr\egroup
      \global\advance\c@equation\m@ne$$\global\@ignoretrue}

\def\@yeqncr{\@ifnextchar [{\@xeqncr}{\@xeqncr[5pt]}}
\makeatother

\begin{document}
\bibliographystyle{tom}

\begingroup
\makeatletter
\patchcmd{\@settitle}{\uppercasenonmath\@title}{\Large}{}{}
\patchcmd{\@setauthors}{\MakeUppercase}{\large}{}{}
\makeatother
\maketitle
\endgroup

\section*{Introduction} \label{intro}
Form methods give a very efficient tool to solve evolutionary problems on
Hilbert space. They were developed by T. Kato \cite{Kat1} and, in slightly
different language by J.L. Lions. In this expository article we give an
introduction based on \cite{AE2}. 
The main point in our approach is that the notion of closability is not needed anymore.
The new setting is particularly efficient
for the Dirichlet-to-Neumann operator and degenerate equations. Besides this we
give several other examples. This presentation starts by an introduction to
holomorphic semigroups. Instead of the contour argument found in the literature,
we  give a more direct argument based on the Hille--Yosida theorem.

\section{The Hille--Yosida Theorem}\label{ch:hille}

A \textit{$C_0$-semigroup} on a Banach space $X$ is a mapping $T\colon (0,\infty)\to
\cL(X)$ satisfying
\begin{eqnarray*}
&&T(t+s)=T(t)T(s)\\
&&\lim\limits_{t\downarrow 0} T(t)x=x\qquad (x\in X) \ .
\end{eqnarray*}
The \textit{generator} $A$ of such a $C_0$-semigroup is defined by
\begin{eqnarray*}
D(A)&:=&\{x\in X:\lim\limits_{t\downarrow 0}\frac{T(t)x-x}{t}  \  \mbox{exists}\}\\
Ax&:=& \lim\limits_{t\downarrow 0}\frac{T(t)x-x}{t} \qquad (x \in D(A)) \ .
\end{eqnarray*}
Thus the domain $D(A)$ of $A$ is a subspace of $X$ and $A\colon D(A)\to X$ is
linear. One can show  that $D(A)$ is dense in $X$. The main interest in
semigroups lies in the associated Cauchy problem
\[
{\rm (CP)} \quad \left\lbrace 
\begin{array}{l@{}l@{}l}
\dot{u}(t)&{}={}& Au(t) \qquad (t>0)\\
u(0)&{}={}&x \ .
\end{array} \right.
\]
Indeed, if $A$ is the generator of a $C_0$-semigroup, then given $x\in X$, the
function $u(t):=T(t)x$ is the unique \textit{mild} solution of $(CP)$; i.e.
\[
u\in C([0,\infty);X) \ , \ \int\limits^t_0 u(s)  \, ds  \in D(A)
\]
for all $t > 0$ and
\begin{eqnarray*}
u(t)&=&x+A\int\limits^t_0 u(s) \, ds \\[0pt]
u(0)&=&x \ .
\end{eqnarray*}
If $x\in D(A)$, then $u$ is a \textit{classical solution}; i.e.\ $u \in
C^1([0,\infty);X)$, $u(t)\in D(A)$ for all $t\ge 0$ and $\dot{u}(t)=Au(t)$ for all
$t>0$.
Conversely, if for each $x\in X$ there exists a unique mild solution of $(CP)$,
then $A$ generates a $C_0$-semigroup \cite[Theorem 3.1.12]{ABHN}. In view of
this characterization of well-posedness, it is of big interest to decide
whether a given operator generates a $C_0$-semigroup.
A positive answer is given by the famous Hille--Yosida Theorem.

\begin{theorem} {\rm (Hille--Yosida (1948))}. \label{thm:1.1}
Let $A$ be an operator on $X$. The following are equivalent.
\begin{enumerate}
 \item[(i)]
$A$ generates a contractive $C_0$-semigroup;
  \item[(ii)]
the domain of $A$ is dense, $\lambda-A$ is invertible for some $($all$)$ $\lambda >0$ and
$\|\lambda (\lambda-A)^{-1}\|\le 1$.
\end{enumerate}
\end{theorem}

Here we  call a semigroup $T$ \textit{contractive} if $\|T(t)\|\le 1$ for all
$t>0$. By
$\lambda-A$ we mean the operator with domain $D(A)$ given by
$(\lambda-A)x:=\lambda x-Ax\quad (x \in D(A))$.
So the condition in (ii) means that $\lambda-A\colon D(A)\to X$  is bijective and
$\|\lambda (\lambda-A)^{-1}x\|\le \|x\|$ for all $\lambda > 0$ and $x\in X$.
If $X$ is reflexive, then this existence of the \textit{resolvent}
$(\lambda-A)^{-1}$ and the contractivity $\|\lambda(\lambda-A)^{-1}\|\le 1$
imply already that the domain is dense \cite[Theorem 3.3.8]{ABHN}.

Yosida's proof is based on the Yosida-approximation: Assuming (ii), one easily
sees that
\[
\lim\limits_{\lambda \to \infty} \lambda (\lambda-A)^{-1}x=x \qquad(x\in D(A)) \ ,
\]
i.e.\ $\lambda (\lambda-A)^{-1}$ converges strongly to the identity as $\lambda
\to \infty$. This implies that
\[
A_\lambda := \lambda A (\lambda-A)^{-1}=\lambda^2(\lambda-A)^{-1}-\lambda
\]
approximates $A$ as $\lambda \to \infty$ in the sense that
\[
\lim\limits_{\lambda \to \infty}A_\lambda x = Ax \qquad (x \in D(A)) \ .
\]
The operator $A_\lambda$ is  bounded, so one may define
\[
e^{tA_\lambda}:=\sum\limits^\infty_{n=0}\frac{t^n}{n!} \, A^n_\lambda
\]
by the power series. Note that $\|\lambda^2 (\lambda-A)^{-1}\| \le \lambda$.
Since
\[
e^{tA_\lambda}=e^{-\lambda t}e^{t \lambda^2(\lambda-A)^{-1}}  \ ,
\]
it follows that
\[
\|e^{tA_\lambda}\| \le e^{-\lambda t} e^{t\| \lambda^2(\lambda-A)^{-1}\|} \le 1
\ .
\]
The key element in Yosida's proof consists in showing that for all $x \in X$ the family
$(e^{tA_\lambda}x)_{\lambda > 0}$ is a
Cauchy net as $\lambda \to \infty$. Then the $C_0$-semigroup generated by $A$
is given by
\[
T(t)x:=\lim\limits_{\lambda \to \infty} e^{tA_\lambda}x \qquad (t>0)
\]
for all $x \in X$.
We will come back to this formula when we talk about holomorphic semigroups.

\begin{remark} \label{rem:1.2} Hille's independent proof is based on Euler's
formula for the
exponential function. Note that putting $t=\frac{1}{\lambda}$ one has
\[
\lambda(\lambda-A)^{-1}=(I-tA)^{-1} \ .
\]
Hille showed that
\[
T(t)x:=\lim\limits_{n\to \infty} (I-\frac{t}{n}A)^{-n}x
\]
exists for all $x \in X$, see \cite[Section~IX.1.2]{Kat1}.
\end{remark}

\section{Holomorphic semigroups}\label{ch:holomorphic}

A $C_0$-semigroup is defined on the real half-line $(0,\infty)$ with values in
$\cL(X)$. It is useful to study when extensions to a sector
\[
\Sigma_\theta :=\{re^{i\alpha}:r>0, \; |\alpha| < \theta \}
\]
for some $\theta \in (0,\pi/2]$ exist. In this section $X$ is a complex Banach
space.

\begin{definition} \label{def2.1}
A $C_0$-semigroup $T$ is called \textit{holomorphic} if there exist $\theta
\in (0,\pi/2]$ and a holomorphic extension
\[
\widetilde{T}\colon \Sigma_\theta \to \cL(X)
\]
of $T$ which is locally bounded; i.e. 
\[
\sup\limits_{\scriptstyle z\in \Sigma_\theta \atop
             \scriptstyle |z| \le 1} \|\widetilde{T}(z)\| < \infty \ .
\]
If $\|\widetilde{T}(z)\|\le 1$ for all $z\in\Sigma_\theta$, then we call $T$ a
\textit{sectorially contractive holomorphic} $C_0$-semigroup (\textit{of angle
$\theta$}, if we want to make precise the angle).
\end{definition}

The holomorphic extension $\widetilde{T}$ automatically has the semigroup property
\[
\widetilde{T}(z_1+z_2)=\widetilde{T}(z_1)\widetilde{T}(z_2) \qquad (z_1,z_2 \in
\Sigma_\theta) \ .
\]
Because of the boundedness assumption it follows that
\[
\lim\limits_{\scriptstyle z\to 0 \atop \scriptstyle z \in \Sigma_\theta} 
\widetilde{T}(z)x=x \qquad (x\in X) \ .
\]
These properties are easy to see. 
Moreover, $\widetilde T$ can be extended continuously
(for the strong operator topology) to the closure of 
$\Sigma_\theta$, keeping these two properties.
In fact, if $x = T(t) y$ for some $t > 0$ and some $y \in X$,
then 
\[
\lim_{w \to z} T(w) x
= \lim_{w \to z} T(w+t) y
= T(z+t) y
\]
exists.
Since the set $ \{ T(t) y : t \in (0,\infty), \; y \in X \} $ is dense 
the claim follows.
In the sequel we will omit the tilde and
denote the extension $\widetilde{T}$ simply by $T$.
We should add a remark on vector-valued holomorphic functions.

\begin{remark}\label{rem2.2}
If $Y$ is a Banach space, $\Omega \subset \C$ open, then a function
$f\colon \Omega\to Y$ is called \textit{holomorphic} if
\[
f^\prime (z)=\lim\limits_{h\to 0}\frac{f(z+h)-f(z)}{h}
\]
exists in the norm of $Y$ for all $z \in \Omega$ and $f^\prime\colon \Omega \to Y$ is
continuous. It follows as in the scalar case that $f$ is analytic. It is
remarkable that holomorphy is the same as weak holomorphy (first observed by
Grothendieck): A function $f\colon \Omega \to Y$  is holomorphic if and only if
\[
y^\prime \circ f\colon \Omega \to \C
\]
is holomorphic for all $y^\prime \in Y'$. In our context the space $Y$ is
$\cL(X)$, the space of all bounded linear operators on $X$ with the operator
norm. If the function $f$ is bounded it suffices to test holomorphy with few
functionals. We say that a subspace $W \subset Y^\prime$ \textit{separates
points} if for all $x\in Y$,
\[
\langle y^\prime,x\rangle = 0 \mbox{ for all } y^\prime \in W \mbox{ implies } x=0
\ .
\]
Assume that $f\colon \Omega \to Y$ is bounded such that $y^\prime \circ f$  is
holomorphic for all $y^\prime \in W$ where $W$  is a  separating subspace of
$Y^\prime$. Then $f$ is holomorphic. This result  is due to \cite{ArN}, see
also \cite[Theorem A7]{ABHN}. In particular, if $Y=\cL(X)$, then a bounded
function $f\colon \Omega \to \cL(X)$ is holomorphic if and only if $\langle
x^\prime,f(\cdot)x\rangle$ is holomorphic for all $x$ in a dense subspace of
$X$ and all $x^\prime$ in a separating subspace of $X^\prime$.
\end{remark}

We recall a special form of Vitali's Theorem (see \cite{ArN}, \cite[Theorem
A5]{ABHN}). 

\begin{theorem} {\rm (Vitali)}. \label{thm:2.3} 
Suppose $\Omega \subset \C$ is connected.
For all $n \in \N$ let $f_n\colon \Omega \to \cL(X)$  be holomorphic,
let $M \in \R$ and suppose that 
\begin{enumerate}
\addtolength{\itemsep}{0.3\baselineskip}
 \item[a)]
$\|f_n(z)\| \le M \qquad$ for all $z \in \Omega$ and $n \in \N$, and;
  \item[b)]
$\Omega_0 := \{ z \in \Omega : \lim_{n \to \infty} f_n(z)x \mbox{ exists for all }
    x \in X \} $
has a limit point in $\Omega$, 
i.e.\ there exist a sequence $(z_k)_{k \in \N}$ in $\Omega_0$ and 
$z_0 \in \Omega$ such that $z_k \neq z_0$ for all $k \in \N$ and 
$\lim\limits_{k\to \infty} z_k=z_0$. 
\end{enumerate}
Then
\[
f(z)x:=\lim\limits_{n\to \infty}f_n(z)x
\]
exists for all $x \in X$ and $z \in \Omega$,
and $f\colon \Omega\to \cL(X)$ is holomorphic.
\end{theorem}

Now we want to give a simple characterization of holomorphic sectorially
contractive semigroups.
Assume that $A$ is a densely defined operator on $X$ such that
$(\lambda-A)^{-1}$ exists and
\[
\|\lambda (\lambda-A)^{-1}\| \le 1 \qquad (\lambda \in \Sigma_\theta) \ ,
\]
where $0<\theta \le \pi/2$. Let $z \in \Sigma_\theta$. Then for all $\lambda > 0$,
\[
(zA)_\lambda = zA_{\frac{\lambda}{z}}
\]
is holomorphic in $z$. For each $z\in \Sigma_\theta$, the operator $zA$ 
satisfies Condition~(ii) of Theorem~\ref{thm:1.1}.
 By the Hille--Yosida Theorem
\[
T(z)x:=\lim\limits_{\lambda \to \infty} e^{(zA)_{\lambda}}x
\]
exists for all $x \in X$ and $z \in \Sigma_\theta$.
Since $z \mapsto e^{(zA)_\lambda}=e^{zA_{\lambda/z}}$ is holomorphic, $T\colon \Sigma_\theta \to
\cL(X)$ is holomorphic by Vitali's Theorem.
If $t>0$, then
\[
T(t)=\lim\limits_{\lambda \to \infty} e^{tA_{\lambda/t}}=T_A(t)
\]
where $T_A$ is the semigroup generated by $A$. Since $T_A(t+s)=T_A(t) T_A(s)$, it
follows from  analytic continuation  that
\[
T(z_1+z_2)=T(z_1)T(z_2) \qquad (z_1,z_2 \in \Sigma_\theta) \ .
\]
Thus $A$ generates a sectorially contractive holomorphic $C_0$-semigroup of
angle $\theta$ on $X$. One sees as above that
\[
T_{zA}(t)=T(zt)
\]
for all $t>0$ and $z\in \Sigma_\theta$. We have shown the following.

\begin{theorem}\label{thm:2.4}
Let $A$ be a densely defined operator on $X$ and $\theta \in (0,\pi/2]$. The
following are equivalent.
\begin{enumerate}
 \item[(i)]
$A$ generates a sectorially contractive holomorphic $C_0$-semigroup of angle
$\theta$;
  \item[(ii)]
$(\lambda-A)^{-1}$ exists for all $\lambda \in \Sigma_\theta$ and
\[
\|\lambda(\lambda-A)^{-1}\|\le 1 \qquad (\lambda \in \Sigma_\theta) \ .
\]
\end{enumerate}
\end{theorem}

We refer to \cite{AEH} for a similar approach to possibly noncontractive holomorphic
semigroups.

\section{The Lumer--Phillips Theorem}\label{ch:Lumer}

Let $H$ be a Hilbert space over $\K=\R$ or $\C$. An operator $A$  on $H$ is
called \textit{accretive} or \textit{monotone} if
\[
\Re (A x | x) \ge 0 \qquad (x \in D(A)) \ .
\]
Based on this notion the following very convenient characterization 
is an easy consequence of the Hille--Yosida Theorem.

\begin{theorem} {\rm (Lumer--Phillips)}. \label{thm3.1}
Let $A$ be an operator on $H$. The following are equivalent.
\begin{enumerate}
 \item[(i)]
$-A$ generates a contraction semigroup; 
 \item[(ii)]
$A$ is accretive and $I+A$ is surjective.
\end{enumerate}
\end{theorem}

For a proof, see \cite[Theorem~3.4.5]{ABHN}.
Accretivity of $A$ can be  reformulated by the condition
\[
\|(\lambda+A) x \| \ge \|\lambda x\| \qquad (\lambda > 0, \; x \in D(A))
\ .
\]
Thus if $\lambda + A$ is surjective, then $\lambda + A$ is  invertible and
$\|\lambda (\lambda+A)^{-1}\| \le 1$.
We also say that $A$ is \textit{$m$-accretive} if Condition (ii) is satisfied.
If $A$ is $m$-accretive and $\K=\C$, then one can easily see that $\lambda + A$
is invertible for all $\lambda \in \C$ satisfying $\Re \lambda > 0$ and
\[
\|(\lambda +A)^{-1}\| \le \frac{1}{\Re \lambda} \ .
\]
Due to the reflexivity of Hilbert spaces, each $m$-accretive operator $A$ is
densely defined (see \cite[Proposition 3.3.8]{ABHN}).
Now we want to  reformulate the Lumer--Phillips Theorem for generators of
semigroups which are contractive on a sector.

\begin{theorem} {(generators of sectorially contractive semigroups)}.
\label{thm3.2}
Let $A$ be an operator on a complex Hilbert space $H$ and let $\theta \in
(0,\frac{\pi}{2})$. The following are equivalent.
\begin{enumerate}
 \item[(i)]
$-A$ generates a holomorphic $C_0$-semigroup which is contractive on the
sector $\Sigma_\theta$;
 \item[(ii)]
$e^{\pm i\theta}A$ is accretive and $I+A$ is surjective.
\end{enumerate}
\end{theorem}

\begin{proof} $(ii)\Rightarrow (i)$.
Since $e^{\pm i\theta}A$ is accretive the operator
$zA$ is accretive for all $z \in
\Sigma_\theta$. Since $(I+A)$ is surjective, the operator $A$ is $m$-accretive.
Thus $(\lambda+A)$ is invertible whenever $\Re \lambda > 0$. Consequently
$(I+zA)=z(z^{-1}+A)$ is invertible for all $z\in\Sigma_\theta$. Thus $zA$ is
$m$-accretive for all $z\in\Sigma_\theta$. Now (i) follows from Theorem
\ref{thm:2.4}.

$(i)\Rightarrow (ii)$.
If $-A$ generates a holomorphic semigroup which is contractive on
$\Sigma_\theta$, then $e^{i\alpha} A$ generates a contraction semigroup for
all $\alpha$ with 
$|\alpha|\le \theta$. Hence $e^{i\alpha} A$ is $m$-accretive whenever $|\alpha| \le
\theta$.
\end{proof}

\section{Forms: the complete case}\label{ch:Forms}

We  recall one of our most efficient tool to solve equations, the Lax--Milgram
lemma, which is just a non-symmetric generalization of the Riesz--Fr\'echet 
representation theorem  from 1905.

\begin{lemma} {\rm (Lax--Milgram (1954))}. \label{lem4.1}
Let $V$ be a Hilbert space over $\K$, where
$\K=\R$ or $\K = \C$, and let $a\colon V\times V \to \K$ be
sesquilinear, continuous and coercive, i.e.
\[
\Re a(u)\ge \alpha \|u\|^2_V \qquad (u\in V)
\]
for some $\alpha > 0$. Let $\varphi \colon V \to \K$ be a continuous anti-linear
form, i.e.\ $\varphi$ is continuous and satisfies
$\varphi(u+v)=\varphi(u)+\varphi(v)$ and 
$\varphi(\lambda u)=\overline{\lambda} \varphi (u)$ for all
$u,v \in V$ and $\lambda \in \K$. Then there is a  unique $u \in V$ such
that
\[
a(u,v)=\varphi (v) \qquad (v\in V) \ .
\]
\end{lemma}

Of course, to say that $a$ is continuous means that
\[
|a(u,v)|\le M\|u\|_V\|v\|_V\qquad (u,v \in V)
\]
for some constant $M$. We let $a(u):=a(u,u)$ for all $u \in V$.

In general, the range condition in the Hille--Yosida Theorem is difficult to
prove. However, if we look at operators associated with a form, the Lax--Milgram
Lemma implies automatically the range condition.
We describe now our general setting in the complete case.
Given is a Hilbert space $V$ over $\K$ with $\K=\R$ or $\K = \C$,
 and a continuous, coercive
sesquilinear form
\[
a\colon V\times V \to \K \ .
\]
Moreover, we assume that $H$ is another Hilbert space over $\K$ and $j\colon V\to H$
is a  continuous linear mapping with dense image.
Now  we associate an operator $A$ on $H$ with the pair $(a,j)$ in the following
way.
Given $x,y \in H$ we say that $x \in D(A)$ and $Ax=y$ if there exists a $u \in
V$ such that $j(u)=x$ and
\[
a(u,v)=(y|j(v))_H \qquad \mbox{ for all } v \in V \ .
\]
We first show that $A$ is well-defined. Assume that there exist $u_1,u_2 \in V$ and 
$y_1,y_2 \in H$ such that
\begin{eqnarray*}
j(u_1)& =& j(u_2) \ , \\
a(u_1,v) & = & (y_1 | j(v))_H  \qquad ( v \in V) \mbox{, and,}\\
a(u_2,v) & = & (y_2|j(v)_H  \qquad ( v \in V ) \ . 
\end{eqnarray*}
Then $a(u_1-u_2,v)=(y_1-y_2 | j(v))_H$ for all $v\in V$. Since $j(u_1-u_2)=0$,
taking $v:=u_1-u_2$ gives $a(u_1-u_2,u_1-u_2)=0$. Since $a$ is coercive, it
follows that $u_1=u_2$. It follows that $(y_1 | j(v))_H=(y_2 | j (v))_H$ for
all $v \in V$. Since $j$ has dense image, it follows that $y_1=y_2$. 

It is clear from the definition that $A\colon D(A)\to H$ is linear. Our main result
is the following generation theorem. We  first assume that $\K=\C$.

\begin{theorem} \label{thm:4.2} {\rm (generation theorem in the complete case)}.
The operator $-A$ generates a sectorially contractive holomorphic
$C_0$-semigroup~$T$. If $a$ is symmetric, then $A$ is selfadjoint.
\end{theorem}

\begin{proof}
Letting $M\ge 0$ be the constant of continuity and $\alpha > 0$ the constant of
coerciveness as before, we have
\[
\frac{| \Im a (v)|}{\Re a (v)} \le \frac{M\|v\|^2_V}{\alpha \|v\|^2_V} =
\frac{M}{\alpha}
\]
for all $v\in V \setminus \{ 0 \} $. 
Thus there exists a $\theta^\prime \in (0,\frac{\pi}{2})$ such
that
\[
a(v)\in \overline{\Sigma_{\theta^\prime}} \qquad (v \in V) \ .
\]
Let $x \in D(A)$. There exists a $u \in V$ such that $x=j(u)$ and $a(u,v)=(Ax |
j(v))_H$ for all $v \in V$. 
In particular, $(Ax | x)_H=a(u)\in \overline{\Sigma_{\theta^\prime}}$. 
It follows that $e^{\pm i\theta}A$ is accretive  where
$\theta=\frac{\pi}{2}-\theta^\prime$. In order to prove the range condition, let
$y\in H$. Consider the form $b\colon V\times V \to \C$ given by
\[
b(u,v)=a(u,v)+(j(u)|j(v))_H \ .
\]
Then $b$ is continuous and coercive. Let $y \in H$. Then
$\varphi(v):=(y|j(v))_H$ defines a continuous anti-linear form $\varphi$ on $V$.
By the Lax--Milgram Lemma~\ref{lem4.1} there exists a unique $u \in V$ such that
\[
b(u,v)=\varphi(v)\qquad (v \in V) \ .
\]
Hence $(y|j(v))_H=a(u,v)+(j(u)|j(v))_H$; i.e.\ $a(u,v)=(y-j(u)|j(v))_H$ for
all $v\in V$. 
This means that $x:= j(u)\in D(A)$ and $Ax=y-x$.
\end{proof}

The result is also valid in real Banach spaces. If $T$ is a $C_0$-semigroup on
a real Banach space $X$, then the $\C$-linear extension $T_{\C}$ of $T$ on the
complexification $X_{\C}:=X\oplus i X$ of $X$ is a $C_0$-semigroup given by
$T_{\C}(t)(x+iy):=T(t)x+iT(t)y$. We call $T$ \textit{holomorphic} if $T_{\C}$ is
holomorphic. The generation theorem above remains true on real Hilbert spaces.

In order to formulate a final result we want also allow a rescaling. Let $X$
be a Banach space over $\K$ and $T$ be a $C_0$-semigroup on $X$ with generator
$A$. Then for all $\omega \in \K$ and $t > 0$ define
\[
T_\omega(t):=e^{\omega t}T(t) \ .
\]
Then $T_\omega$ is a $C_0$-semigroup whose generator is $A+\omega$. Using this we obtain
now the following general generation theorem in the complete case.

Let $V,H$ be Hilbert spaces over $\K$ and $j\colon V\to H$ linear with dense image.
Let $a\colon V\times V\to \K$ be sesquilinear and continuous. We call the form $a$
$j$-\textit{elliptic} if there exist $\omega \in \R$ and $\alpha > 0$ such that
\begin{equation} \label{equ:4.1}
\Re a(u)+\omega \|j(u)\|_H^2 \ge \alpha \|u\|_V^2
\qquad (u \in V) \,
\end{equation}
Then we define the operator $A$ associated with $(a,j)$ as follow.
Given $x,y \in H$ we say that $x \in \D(A)$ and $Ax=y$ if there exists a $u \in
V$ such that $j(u)=x$ and
\[
a(u,v)=(y|j(v))_H \qquad \mbox{ for all } v \in V \ .
\]

\begin{theorem}\label{thm:4.3}
The operator defined in this way is well-defined. Moreover, $-A$ generates a
holomorphic $C_0$-semigroup on $H$.
\end{theorem}

\begin{remark} \label{rem:4.4}
The form $a$ satisfies Condition (\ref{equ:4.1}) if and only if the form
$a_\omega$ given by
\[
a_\omega(u,v)=a(u,v)+\omega(j(u)|j(v))_H
\]
is coercive. If $T_\omega$ denotes the semigroup associated with $(a_\omega,j)$ and
$T$ the semigroup associated with $(a,j)$, then
\[
T_\omega(t)=e^{-\omega t}T(t)\qquad (t > 0) 
\]
as is easy to see.
\end{remark}

\section{The Stokes Operator}\label{ch:stokes}

In this section we show as an example that the Stokes operator is selfadjoint
and generates a holomorphic $C_0$-semigroup. The following approach is due to
Monniaux \cite{Mon1}.
Let $\Omega \subset \R^d$ be  a bounded open set. We  first discuss the
Dirichlet Laplacian.

\begin{theorem} \label{thm:5.1} {\rm (Dirichlet Laplacian)}.
Let  $H=L^2(\Omega)$ and define the operator $\Delta^D$ on $L^2(\Omega)$ by
\begin{eqnarray*}
D(\Delta^D)&=&\{u\in H^1_0(\Omega):\Delta u \in L^2(\Omega)\} \\
\Delta^D u &:=& \Delta u \ .
\end{eqnarray*}
Then $\Delta^D$ is selfadjoint  and generates a holomorphic $C_0$-semigroup on
$L^2(\Omega)$.
\end{theorem}

\begin{proof}
Define $a\colon H^1_0(\Omega)\times H^1_0(\Omega)\to \R$ by
$a(u,v)=\int\limits_\Omega \nabla u \nabla v$. Then $a$ is clearly continuous.
Poincar\'e's inequality says that $a$ is coercive. Consider the injection $j$
of $H^1_0(\Omega)$ into $L^2(\Omega)$. Let $A$ be the operator associated with
$(a,j)$. We show that $A=-\Delta^D$. In fact, let $u \in D(A)$ and write $f = Au$. Then
$\int\limits_\Omega \nabla u \nabla v =\int\limits_\Omega f v$ for all $v \in
H^1_0(\Omega)$. Taking in particular $v \in C^\infty_c(\Omega)$ we see that
$-\Delta u = f$. Conversely, let $u \in H^1_0(\Omega)$ be such that $f:=-\Delta u
\in L^2(\Omega)$. Then $\int\limits_\Omega f \varphi = \int\limits_\Omega
\nabla u \nabla \varphi=a(u,\varphi)$ for all $\varphi \in C^\infty_c(\Omega)$.
This is just the definition of the weak partial derivatives in $H^1(\Omega)$.
Since $C^\infty_c(\Omega)$ is dense in $H^1_0(\Omega)$, it follows that
$\int\limits_\Omega fv=a(u,v)$ for all $v \in H^1_0(\Omega)$. Thus $u \in D(A)$
and $Au=f$.
\end{proof}

For our treatment of the Stokes operator it will be useful to consider the
Dirichlet Laplacian also in $L^2(\Omega)^d=L^2(\Omega)\oplus \ldots \oplus
L^2(\Omega)$.

\begin{theorem} \label{thm:5.2}
Define the symmetric form
$a\colon H^1_0(\Omega)^d \times H^1_0(\Omega)^d \to \R$ by
\[
a(u,v)= \int\limits_\Omega \nabla u \nabla v 
:= \sum\limits^d_{j=1} \int\limits_\Omega \nabla u_j \nabla v_j \ ,
\]
where $u=(u_1,\ldots, u_d)$. Then $a$ is continuous and coercive. 
Moreover, let
$j\colon H^1_0(\Omega)^d\to L^2(\Omega)^d$ be the identity. The operator $A$
associated with $(a,j)$ on $L^2(\Omega)^d$ is given by
\begin{eqnarray*}
D(A)&=&\{u\in H^1_0(\Omega)^d: \Delta u_j \in L^2(\Omega) \mbox{ for all } j \in \{ 1,\ldots, d \} \} \ , \\
Au&=& (-\Delta u_1,\ldots, - \Delta u_d)=:-\Delta u \ .
\end{eqnarray*}
We call $\Delta^D:=-A$ the \textit{Dirichlet Laplacian} on $L^2(\Omega)^d$.
\end{theorem}

In order to define the Stokes operator we need some preparation. Let
$\cD(\Omega):=C^\infty_c(\Omega)^d$ and let $\cD_0(\Omega):=\{\varphi \in
\cD(\Omega): \div \varphi = 0\}$, where $\div \varphi=\partial_1\varphi_1
+\ldots + \partial_d \varphi_d$ and $\varphi=(\varphi_1,\ldots,\varphi_d)$.
By $\cD(\Omega)^\prime$ we denote the dual space of $\cD(\Omega)$ (with the
usual topology).
Each element $S$ of $\cD(\Omega)^\prime$ can be written in a unique way as
$S=(S_1,\ldots,S_d)$ with $S_j\in C^\infty_c(\Omega)^\prime$ so that
\[
\langle S,\varphi \rangle = \sum\limits^d_{j=1} \langle S_j,\varphi_j \rangle
\]
for all $\varphi=(\varphi_1,\dots,\varphi_d)\in \cD(\Omega)$. 

We say that $S\in H^{-1}(\Omega)$ if there exists a constant $c \geq 0$ such that
\[
|\langle S,\varphi\rangle | \le c \, (\int | \nabla \varphi|^2)^{\frac{1}{2}}
\qquad (\varphi \in \cD(\Omega))
\]
where $| \nabla \varphi|^2=| \nabla \varphi_1|^2 + \ldots + | \nabla
\varphi_d|^2$.
For the remainder of this section we assume that $\Omega$ has Lipschitz
boundary. We need the following result (see \cite[Remark 1.9, p. 14]{Temam}).

\begin{theorem}  \label{thm:5.3}
Let $T\in H^{-1}(\Omega)$. The following are equivalent.
\begin{enumerate}
 \item[(i)]
$\langle T,\varphi \rangle = 0$ for all $\varphi \in \cD_0(\Omega)$;
 \item[(ii)]
there exists a $p\in L^2(\Omega)$ such that $T=\nabla p$.
\end{enumerate}
\end{theorem}

Note that Condition (ii) means that
\[
\langle T,\varphi \rangle 
= \sum\limits^d_{j=1} \langle \partial_j p, \varphi_j \rangle
= -\sum\limits^d_{j=1} \langle p,\partial_j \varphi_j \rangle
= - \langle p, \div \varphi \rangle \ .
\]
Now the implication (ii) $\Rightarrow$ (i) is obvious. We omit the other
implication.

Consider the real Hilbert space $L^2(\Omega)^d$ with scalar product
\[
(f|g)=\sum\limits^d_{j=1}(f_j | g_j)_{L^2(\Omega)} = \sum\limits^d_{j=1}
\int\limits_\Omega f_j g_j \ .
\]
We denote by
\[
H:=\cD_0(\Omega)^{\bot\bot}=\overline{\cD_0(\Omega)}
\]
the closure of $\cD_0(\Omega)$ in $L^2(\Omega)^d$. We call $H$ the space of all
\textit{divergence free vectors} in $L^2(\Omega)^d$.
The orthogonal projection $P$ from $L^2(\Omega)^d$ onto $H$ is called the
\textit{Helmholtz projection}.
Now let $V$ be the closure of $\cD_0(\Omega)$ in $H^1(\Omega)^d$. Thus $V \subset
H^1_0(\Omega)^d$ and $\div u = 0$ for all $u \in V$. One can actually show that 
\[
V=\{ u  \in H^1_0(\Omega)^d : \div v = 0 \} \ .
\]
We define the form $a\colon V\times V \to \R$ by
\[
a(u,v)=\sum\limits^d_{j=1}(\nabla u_j | \nabla v_j)_{L^2(\Omega)} \quad
(u=(u_1,\dots,u_d),v=(v_1,\ldots,v_d) \in V) \ .
\]
Then $a$ is continuous and coercive.
The space $V$ is dense in $H$ since it contains $\cD_0(\Omega)$. We consider the
identity $j\colon V\to H$. Let $A$ be the operator associated with $(a,j)$. Then $A$
is selfadjoint and $-A$ generates a holomorphic $C_0$-semigroup.
The operator can be described as follows.

\begin{theorem} \label{thm:5.4}
The operator $A$ has the domain
\[
D(A)=\{u\in V:\exists\, \pi \in L^2(\Omega) \mbox{ such that } - \Delta u +
\nabla \pi \in H\}
\]
and is given by
\[
Au = - \Delta u+\nabla \pi \ ,
\]
where $\pi \in L^2(\Omega)$ is such that $- \Delta u + \nabla \pi \in H$.
\end{theorem}

If $u\in H^1_0(\Omega)^d$, then $\Delta u \in H^{-1}(\Omega)$. In fact, for all
$\varphi \in \cD(\Omega)$,
\[
|\langle - \Delta u,\varphi \rangle |
=|\mbox{$-$}\langle u,\Delta \varphi \rangle|
=|\sum\limits^d_{j=1}\int\limits_\Omega \nabla u_j \nabla \varphi_j|
\le \|u\|_{H^1_0(\Omega)^d} \|\varphi\|_{H^1_0(\Omega)^d} \ .
\]

\begin{proof}[Proof of Theorem \ref{thm:5.4}]
Let $u \in D(A)$ and write $f = Au$. Then $f \in H$, $u \in V$ and $a(u,v)=(f|v)_H$ for all
$v\in V$. Thus, the distribution $-\Delta u \in H^{-1}(\Omega)$ coincides with
$f$ on $\cD_0(\Omega)$. By Theorem \ref{thm:5.3} there exists a 
$\pi \in L^2(\Omega)$ such that $-\Delta u + \nabla \pi = f$. 
Conversely, let $u \in V$, $f \in H$, $\pi \in L^2(\Omega)$ and suppose that 
$-\Delta u + \nabla \pi=f$ in $\cD(\Omega)^\prime$. Then
for all $\varphi \in \cD_0(\Omega)$,
\[
a(u,\varphi)=\int\limits_\Omega \nabla u \nabla \varphi = \int\limits_\Omega
\nabla u \nabla \varphi + \langle \nabla \pi,\varphi \rangle
= (f|\varphi)_{L^2(\Omega)^d} \ .
\]
Since $\cD_0(\Omega)$ is dense in $V$, it follows that
$a(u,\varphi)=(f|\varphi)_{L^2(\Omega)^d}$ for all $\varphi \in V$. Thus, $u
\in D(A)$ and $Au=f$.
\end{proof}

The operator $A$ is called the \textit{Stokes operator}. We refer to
\cite{Mon1} for this approach and further results on the Navier--Stokes
equation.
We conclude this section by giving an example where $j$ is not injective.
Further examples will be seen in the sequel.

\begin{proposition} \label{prop:5.5}
Let $\widetilde{H}$ be a Hilbert space and $H\subset \widetilde{H}$ a closed subspace.
Denote by $P$ the orthogonal projection onto $H$. Let $\widetilde{V}$ be a Hilbert
space which is continuously and densely embedded into $\widetilde{H}$ and let
$a\colon \widetilde{V}\times \widetilde{V} \to \R$ be a continuous, coercive form. Denote by
$A$ the operator on $\widetilde{H}$ associated with $(a,j)$ where $j$ is the
injection of $\widetilde{V}$ into $\widetilde{H}$ and let $B$ be the operator on $H$
associated with $(a,P \circ j)$. Then
\begin{eqnarray*}
D(B)&=&\{Pw:w\in D(A) \mbox{ and } Aw \in H\} \ , \\
BPw&=&Aw \qquad (w \in D(A), \; Aw \in H) \ .
\end{eqnarray*}
\end{proposition}

This is easy to see. In the context considered in this section we obtain the
following example.

\begin{example}\label{ex:5.6}
Let
$\widetilde{H}=L^2(\Omega)^d$, $H=\overline{\cD_0(\Omega)}$ and $\widetilde{V}
:=H^1_0(\Omega)^d$.
Define $a \colon \widetilde V \times \widetilde V \to \R$ by
\[
a(u,v)=\int\limits_\Omega \nabla u \nabla v \ .
\]
Moreover, define $j \colon \widetilde{V}\to \widetilde{H}$ by $j(u)=u$. 
Then the operator associated with $(a,j)$ is
$A=-\Delta^D$ as we have seen in Theorem \ref{thm:5.2}. Now let $P$ be the
Helmholtz projection and $B$ the operator associated with $(a,P)$. Then
\begin{eqnarray*}
D(B) = \{u\in H & :& \exists\, \pi \in L^2(\Omega) \mbox{ such that }  \\*
& & u+\nabla \pi \in D(\Delta^D)
\mbox{ and } \Delta (u+\nabla \pi) \in H\}
\end{eqnarray*}
and 
\[
Bu = - \Delta (u+\nabla \pi) \ ,
\]
if $\pi \in L^2(\Omega)$ is such that  $u+\nabla \pi \in D(\Delta^D)$
and $\Delta (u+\nabla \pi) \in H$.
This follows directly from Proposition \ref{prop:5.5} and Theorem~\ref{thm:5.3}.
Thus, the operator $B$ is selfadjoint and generates a holomorphic semigroup.
\end{example}

\section{From forms to semigroups: the incomplete case}

In the  preceding sections we considered forms which were defined on a Hilbert
space $V$. Now we want to study a purely algebraic condition
considering forms whose domain is an arbitrary vector space.
 At first we
consider the complex case.
Let $H$  be a complex Hilbert  space. A \textit{sectorial form} on $H$ is a
sesquilinear form
\[
a\colon D(a)\times D(a)\to \C \ ,
\]
where $D(a)$ is a vector space,
together with a linear mapping $j\colon D(a)\to H$ with dense image such that there
exist $\omega \ge 0$ and $\theta \in (0,\pi/2)$ such that
\[
a(u)+\omega \|j(u)\|^2_H \in \Sigma_\theta \qquad (u \in D(a)) \ .
\]
If $\omega=0$, then we call the form \textit{$0$-sectorial}.
To  a sectorial form, we associate an operator $A$ on $H$ by defining for all $x,y \in H$
that  $x \in D(A)$ and $Ax=y :\Leftrightarrow$ there exists a sequence 
$(u_n)_{n \in \N}$ in $D(a)$ such that
\begin{enumerate}
\addtolength{\itemsep}{0.3\baselineskip}
 \item[a)] 
$\lim\limits_{n\to \infty}j(u_n)=x$ in $H$;
 \item[b)]
$\sup\limits_{n\in \N}\Re a (u_n) < \infty$, and;
 \item[c)]
$\lim\limits_{n\to \infty}a(u_n,v)=(y|j(v))_H$ for all $v \in D(a)$.
\end{enumerate}

It is part of the next theorem that the operator $A$ is well-defined
(i.e.\ that $y$ depends only on $x$ and not on the choice of the sequence
satisfying a), b) and c)).
We only consider single-valued operators in this article.

\begin{theorem}\label{thm:6.1}
The operator $A$ associated with a sectorial form $(a,j)$ is well-defined and $-A$
generates a holomorphic $C_0$-semigroup on $H$.
\end{theorem}

The proof of the theorem consists in a reduction to the complete case by
considering an appropriate completion of $D(a)$. Here it is important that in
Theorem \ref{thm:4.2} a non-injective mapping $j$ is allowed. 
For a proof we refer to \cite[Theorem~3.2]{AE2}.

If $C\subset H$ is  a closed convex set, we say that $C$ is \textit{invariant}
under a semigroup $T$ if
\[
T(t)C\subset C \qquad (t > 0) \ .
\]
Invariant sets are important to study positivity, $L^\infty$-contractivity, and
many more properties. If the semigroup is associated with a form, then the
following criterion, \cite[Proposition~3.9]{AE2}, is convenient.

\begin{theorem} {\rm (invariance)}. \label{thm:6.2}
Let $C\subset H$ be a  closed convex set and let $P$ be the orthogonal
projection onto $C$. 
Then the semigroup $T$ associated with a sectorial form $(a,j)$
on $H$ leaves $C$ invariant if and only if for each $u\in D(a)$ there exists a
sequence $(w_n)_{n\in\N}$ in $D(a)$ such that
\begin{enumerate}
\addtolength{\itemsep}{0.3\baselineskip}
 \item[a)] 
$\lim\limits_{n\to \infty} j(w_n)=Pj(u)$ in $H$;
 \item[b)]
$\limsup\limits_{n\to \infty} \Re a (w_n,u-w_n) \ge 0$, and;
 \item[c)]
$\sup\limits_{n\in \N} \Re a (w_n) < \infty$.
\end{enumerate}
\end{theorem}

\begin{corollary} \label{cor:6.3}
Assume that for each $u\in D(a)$, there exists a $w\in D(a)$ such that
\[
j(w)=Pj(u)
\qquad \mbox{ and } \qquad 
\Re a(w,u-w)\ge 0 \ . 
\]
Then $T(t) C\subset C$ for all $t > 0$.
\end{corollary}

In this section we want to use the invariance criterion to prove a generation
theorem in the incomplete case which is valid in real Hilbert spaces.
Let $H$ be a real Hilbert space. A \textit{sectorial} form on $H$ is a bilinear
mapping
\[
a\colon D(a)\times D(a)\to \R \ ,
\]
where $D(a)$ is a real vector space, together with a linear mapping $j\colon D(a)\to
H$ with dense image such that there are $\alpha,\omega \ge 0$ such that 
\begin{eqnarray*}
|a(u,v)-a(v,u)| 
& \le & \alpha (a(u)+a(v))+\omega(\|j(u)\|^2_H + \|j(v)\|^2_H)  \\*
& & \hspace*{46mm}
\qquad (u,v \in D(a)) \ .  
\end{eqnarray*}
It is easy to see that the form $a$ is sectorial on the real space $H$
if and only if the sesquilinear extenion $a_\C$ of $a$ to the 
complexification of $D(a)$ together with the $\C$-linear extension of $j$ 
is sectorial in the sense formulated in the beginning of this section.

To such a sectorial form $(a,j)$ we
associate an operator $A$ on $H$ by defining for all $x,y \in H$ that 
$x \in D(A)$ and $Ax=y:\Leftrightarrow$ there exists a sequence
$(u_n)$ in $D(a)$ satisfying
\begin{enumerate}
\addtolength{\itemsep}{0.3\baselineskip}
 \item[a)] 
$\lim\limits_{n\to \infty} j(u_n)=x$ in $H$;
 \item[b)]
$\sup\limits_{n\in\N}a(u_n)<\infty$, and;
 \item[c)]
$\lim\limits_{n\to \infty} a(u_n,v)=(y|j(v))_H$ for all $v\in D(a)$.
\end{enumerate}
Then the following holds.

\begin{theorem} \label{thm:6.4}
The operator $A$ is well-defined and $-A$ generates a holomorphic $C_0$-semigroup
on $H$.
\end{theorem}

\begin{proof}
Consider the complexifications $H_{\C}=H\oplus i H$ and
$D(a_{\C}):=D(a)+iD(a)$. Letting
\[
a_{\C} (u,v):=a(\Re u, \Re v) + a(\Im u, \Im v) + i (a(\Re u, \Im v)+a (\Im u,
\Re v))
\] 
for all $u=\Re u +i \Im u,v=\Re v+i \Im v \in D(a_{\C})$.
Then $a_{\C}$ is a sesquilinear 
form. 
Let $J\colon D(a_{\C})\to H_{\C}$ be the $\C$-linear
extension of $j$. Let
\[
b(u,v)=a_{\C}(u,v)+\omega (J(u) | J(v))_{H_{\C}}\qquad  (u,v \in D(a_{\C}))
\ .
\]
Then 
\begin{eqnarray*}
\Im b(u)&=&a(\Im u, \Re u)-a(\Re u, \Im u),\\
\Re b(u)&=&a (\Re u)+a(\Im u)+\omega(\|j(\Re u)\|^2_H + \|j(\Im u)\|^2_H) \ .
\end{eqnarray*}
The assumption implies that there is a $c > 0$ such that 
$|\Im b(u)| \le c \Re b(u)$ for all $u \in D(a_{\C})$.
Consequently,
$b(u)\in \overline{\Sigma_\theta}$, where $\theta = \arctan c$.
Thus the operator $B$ associated with $b$  generates a $C_0$-semi\-group $S_{\C}$
on
$H_{\C}$. It follows from Corollary \ref{cor:6.3} that $H$ is invariant. The 
part $A_\omega$ of $B$ in $H$ is the generator of $S$, where 
$S(t):=S_{\C}(t)_{|_{H}}$.   
 It is easy  to see that $A_\omega-\omega=A$.
\end{proof}

\begin{remark}\label{rem:6.5}
It is remarkable, and important for some applications, that Condition b) 
in Theorem~\ref{thm:6.1} as well as in Theorem~\ref{thm:6.4}
may be replaced
by
\[
\lim\limits_{n,m\to \infty} a(u_n-u_m)=0  \ . \leqno{{\rm b}^\prime)}
\]
\end{remark}

For later purposes we carry over the invariance criterion Theorem
\ref{thm:5.3} to the real case.

\begin{corollary} \label{cor:6.6}
Let $H$ be a real Hilbert space and $(a,j)$ a sectorial form on $H$ with
associated semigroup $T$. Let $C \subset H$ be a closed convex  set and $P$ the
orthogonal  projection onto $C$. Assume that for each $u\in D(a)$ there exists a
$w \in D(a)$ such that
\[
j(w)=Pj(u) \qquad \mbox{ and } \qquad a (w,u-w)  \ge 0 \ .
\]
Then $T(t)C\subset C$  for all $t > 0$.
\end{corollary}

We want to formulate a special case of invariance. An operator $S$ on a space
$L^p(\Omega)$ is called
\begin{center}
\begin{tabular}{r@{}l}
{\em positive} & \ if $\Big( f \geq 0$ a.e.\ implies $S f \geq 0$ a.e.$\Big)$ and  \\[5pt]
{\em submarkovian} & \ if $\Big( f \leq \one$ a.e.\ implies $S f \leq \one$ a.e.$\Big)$ .
\end{tabular}
\end{center}
Thus, an operator $S$ is submarkovian if and only if it is positive and
$\|Sf\|_\infty \le \|f\|_\infty$ for all $f \in L^2 \cap L^\infty$.

\begin{proposition} \label{prop:6.7}
Consider the real space $H=L^2(\Omega)$ and a sectorial form $a$ on $H$. Assume
that for each $u \in D(a)$ one has $u\wedge \one \in D(a)$ and
\[
a(u\wedge \one, (u-\one)^+)\ge 0 \ .
\]
Then the semigroup $T$ associated with $a$ is  submarkovian.
\end{proposition}

\begin{proof}
The set $C:=\{u\in L^2(\Omega):u\le \one \mbox{ a.e.}\}$ is closed and convex.
The  orthogonal projection $P$ onto $C$ is given by $Pu=u\wedge \one$. Thus
$u-Pu=(u-\one)^+$ and the result follows from Corollary \ref{cor:6.3}.
\end{proof}

We conclude this section by some references to the literature. In many text
books, for example \cite{Dav2}, \cite{Kat1}, \cite{MR}, \cite{Ouh5},
\cite{Tan} one finds the notion
of a
sectorial form $a$
on a complex Hilbert space $H$. By this one understands a sesquilinear form
$a\colon D(a)\times D(a)\to \C$ where $D(a)$ is a dense subspace  of $H$ such that
there are $\theta \in (0,\pi/2)$ and $\omega \ge 0$ such that 
$a(u)+\omega \|u\|^2_H\in \overline{\Sigma_\theta}$ for all $u \in D(a)$. Then
\[
\|u\|_a:=(\Re a(u)+ (\omega+1)\|u\|^2_H)^{1/2}
\]
defines a  norm on $D(a)$.  The form is called \textit{closed} if $D(a)$ is
complete for this norm.
This corresponds  to our complete  case with $V=D(a)$ and $j$ the identity. If
the form is not closed, then one may consider the completion $V$ of $D(a)$.
Since the injection $D(a)\to H$ is continuous for the norm $\| \ \|_a$, it has
a continuous extension $j\colon V\to H$. This extension  may be injective or not. The
form is called \textit{closable} if $j$ is injective. In the literature only for
closable forms generation theorems are given, see \cite{AE2} for  precise
references. The results above show that the notion of closability is not
needed.

There is a unique correspondence between sectorially quasi contractive 
holomorphic  semigroups and closed sectorial forms (see \cite[Theorem~VI.2.7]{Kat1}). 
One  looses uniqueness if one considers forms which are merely
closable or
in our general setting if one allows arbitrary maps $j\colon D(a)\to H$ with dense  image.
However, examples show that in many cases a natural operator is obtained by
this general framework.

\section{Degenerate diffusion}

In this section  we use our tools to show that degenerate elliptic operators 
generate holomorphic semigroups on the real space $L^2(\Omega)$. We start with a
$1$-dimensional
example.

\begin{example} (degenerate diffusion in dimension 1). \label{ex:7.1}
Consider the real Hilbert space $H=L^2(a,b)$, where 
$-\infty \le a<b\le\infty$, and let
$\alpha,\beta,\gamma \in L^\infty_{\loc}(a,b)$ be real coefficients. We assume 
that there is a $c_1 \geq 0$ such that 
\[
\gamma^- \in L^\infty(a,b) \mbox{ and } \beta^2(x) \le c_1 \cdot \alpha(x) \qquad (x \in (a,b)) \ .  
\]
\end{example}

We define the bilinear form $a$ on $L^2(a,b)$ by
\[
a(u,v)=\int\limits^b_a \Big( \alpha(x)u^\prime (x)v^\prime (x)  + \beta
(x)u^\prime (x)v(x)+\gamma(x)u(x)v(x) \Big) \, dx
\]
with domain
\[
D(a)=H^1_c(a,b) \ .
\]
We choose $j \colon H^1_c(a,b) \to L^2(a,b)$ to be the identity map.
Then the form $a$ is \textit{sectorial}, i.e.\ there exist constants $c,\omega \ge 0$,
 such that
\begin{eqnarray*}
|a(u,v)-a(v,u)| & \le & c(a(u)+a(v))+\omega (\|u\|^2_{L^2}+\|v\|^2_{L^2}) \\*
& & \hspace*{46mm}  \qquad (u,v \in D(a))  \ .
\label{eS7;1}
\end{eqnarray*}
\begin{proof}
We  use Young's inequality
\[
|xy|\le \varepsilon x^2+\frac{1}{4\varepsilon}y^2
\]
twice. 
Let $u,v \in D(a)$.
On one hand we have for all $\delta > 0$,
\begin{eqnarray*}
|a(u,v)-a(v,u)|&=& |\int\limits^b_a \beta (u^\prime v- uv^\prime)| \\
&\le& \int\limits^b_a \delta \beta^2(u^{\prime 2}+v^{\prime 2})+\frac{1}{4\delta} (u^2+v^2) \ .
\end{eqnarray*}
On the other hand, for all $c,\omega,\varepsilon > 0$ one has
\begin{eqnarray*}
\lefteqn{
c(a(u )+a(v)) + \omega(\|u\|^2_H+\|v\|^2_H)
} \hspace{10mm}  \\*
& = & \int\limits^b_a c\alpha(u^{\prime 2}+v^{\prime 2}) + c\beta(u^\prime u+v^\prime
v)+(c \gamma + \omega)(u^2+v^2)   \\
& \geq & \int\limits^b_a (c\alpha-\varepsilon \beta^2)(u^{\prime 2}+v^{\prime 2})  -  c^2
\frac{1}{4 \varepsilon}(u^2+v^2)+(c\gamma + \omega)(u^2+v^2) \\
& \ge & \int\limits^b_a (c\alpha-\varepsilon \beta^2)(u^{\prime 2}+v^{\prime
2}) + (\omega-c\|\gamma^-\|_{L^\infty}-\frac{c^2}{4\varepsilon})(u^2+v^2)  \ .
\end{eqnarray*}
Therefore (\ref{eS7;1}) is valid if 
$(c\alpha-\varepsilon \beta^2)\ge \delta \beta^2$ and 
$(\omega-c\|\gamma^-\|_{L^\infty}-\frac{c^2}{4\varepsilon})\ge \frac{1}{4\delta}$.
Since $\beta^2 \le c_1 \alpha$ one can find $\delta,\varepsilon,c,\omega$ such that the
conditions are satisfied.
\end{proof}

As a consequence, letting $A$ be the operator associated with $(a,j)$, we know that
$-A$ generates a holomorphic $C_0$-semigroup $T$ on $L^2(\Omega)$. Moreover,
$T$ is submarkovian.

The condition $\beta^2 \le c_1\alpha$ shows in particular that $\{x\in
(a,b):\alpha(x)=0\} \subset \{x\in (a,b):\beta(x)=0\}$.
This inclusion is a natural hypothesis,  since in general an operator of the form 
$\beta u^\prime$ does not generate a holomorphic  semigroup.

A special case is the \textit{Black--Scholes Equation}
\[
u_t+\frac{\sigma^2}{2}x^2 u_{xx}+ r x u_x-ru=0  \ .
\]
This one obtains by choosing $H=L^2(0,\infty)$,
\[
a(u,v)
=\int\limits^\infty_0(\frac{\sigma^2}{2}x^2 u^\prime v^\prime
+(\sigma^2-r)x u^\prime v+ ruv)  \]
and  $D(a)=H^1_c(0,\infty)$.

It is not difficult to extend the example  above to higher dimensions.

\begin{example} \label{ex:7.2}
Let $\Omega \subset \R^d$  be open and for all $i,j \in \{ 1,\ldots,d \} $
let $a_{ij},b_j,c \in
L^\infty_{\loc}(\Omega)$ be real coefficients. 
Assume 
$c^- \in L^\infty(\Omega)$, $a_{ij}=a_{ji}$  and there exists a $c_1 > 0$ such that 
\[
c_1A(x)-B^2(x) \mbox{ is positive semidefinite}
\]
for almost all $x \in \Omega$,  where
\[
A(x)=(a_{ij}(x)) \mbox{ and } B(x)=\diag (b_1(x),\ldots,b_d(x)) \ .
\]
Define the form $a$ on $L^2(\Omega)$ by
\[
a(u,v)=\int\limits_\Omega
\Big( \sum\limits^d_{i,j=1}a_{ij} (\partial_i u) (\partial_j v)
+\sum\limits^d_{j=1}b_j (\partial_ju) v+cuv \Big)
\]
with domain
\[
D(a)=H^1_c(\Omega) \ .
\]
Then $a$ is sectorial. The associated semigroup $T$ on $L^2(\Omega)$ is
submarkovian.
\end{example}

This and the previous example incorporate Dirichlet boundary conditions. In
the next one we consider a degenerate elliptic operator with Neumann boundary
conditions.

\begin{example} \label{ex:7.3}
Let $\Omega \subset \R^d$ be an open, possibly unbounded subset of $\R^d$. 
For all $i,j \in \{ 1,\ldots,d \} $ let
$a_{ij}\in L^\infty(\Omega)$ be \textit{real} coefficients and assume that 
there exists a $\theta \in (0,\pi/2)$ such that 
\[
\sum\limits^d_{i,j=1} a_{ij}(x) \xi_i \overline{\xi_j} \in \overline{\Sigma_\theta} \qquad (\xi \in
\C^d, \; x \in \Omega) 
\ .
\]
Consider the form $a$ on $L^2(\Omega)$ given by
\[
a(u,v)=\int_\Omega \sum\limits^d_{i,j=1}a_{ij} (\partial_iu)(\partial_jv)  
\]
with domain $D(a)=H^1(\Omega)$.
Then $a$ is sectorial. Let $T$ be the associated semigroup. Our
criteria show right away that $T$ is submarkovian. It is remarkable that even
\[
T_\infty(t)\one_\Omega =\one_\Omega \qquad (t > 0) \ .
\]
For bounded $\Omega$ this is easy to prove, but otherwise more sophisticated
tools
are needed (see \cite[Corollar~4.9]{AE2}).
Note that $T$ extends consistently to semigroups $T_p$ on $L^p(\Omega)$
for all $p \in [1,\infty]$, 
where $T_p$ is strongly continuous for all $p<\infty$ and $T_\infty$
is the adjoint of a strongly continuous semigroup on $L^1(\Omega)$.
\end{example}

We want to  add an abstract result which shows that our solutions are some kind
of \textit{viscosity solutions}. This is illustrated particularly well in the
situation of Example \ref{ex:7.3}.

\begin{proposition} {\rm (\cite[Corollary 3.9]{AE2})}. \label{prop:7.4}
Let $V,H$ be real Hilbert spaces such that
$V \underset{d}{\hookrightarrow} H$.
Let $j \colon V \to H$ be the identity map.
Let $a\colon V\times V\to \R$ be continuous and sectorial. Assume that $a(u)\ge 0$
for all $u \in V$. Let $b\colon V\times V \to \R$ be continuous and coercive. Then
for each $n\in\N$ the form
\[
a+\frac{1}{n}b\colon V\times V \to \R
\]
is continuous and coercive. 
Let $A_n$ be the operator associated with $(a+\frac{1}{n}b, j)$ and $A$ with $(a,j)$. Then
\[
\lim_{n \to \infty} (A_n+\lambda)^{-1}f =  (A+\lambda)^{-1}f  \mbox{ in } H
\]
for all $f \in H$ and $\lambda > 0$. Moreover, denoting by $T_n$ and $T$ the
semigroup generated by $-A_n$ and by $-A$ one has
\[
\lim\limits_{n\to \infty}T_n(t)f=T(t)f \mbox{ in } H
\]
for all $f \in H$.
\end{proposition}

The point in the result is that the form $a$ is merely sectorial and may be
degenerate. For instance, in Example \ref{ex:7.3} $a_{ij}(x)=0$ is allowed. If
we perturb by the Laplacian, we obtain a coercive form
\[
a_n\colon H^1(\Omega)\times H^1(\Omega) \to \R
\]
given by
\[
a_n(u,v)=a(u,v)+\frac{1}{n}\int\limits_\Omega \nabla u \nabla v \ .
\]
Then Proposition \ref{prop:7.4} says that in the
situation of Example \ref{ex:7.3} for this perturbation
one has $\lim_{n \to \infty} (A_n+\lambda)^{-1}f = (A+\lambda)^{-1}f$  in
$L^2(\Omega)$ for all $f \in L^2(\Omega)$.

\section{The Dirichlet-to-Neumann operator}

The following example shows how the general setting involving non-injective $j$
can be used. It is taken from \cite{AE3} where also the interplay between
trace properties and the semigroup  generated by the Dirichlet-to-Neumann
operator is studied.
Let $\Omega \subset \R^d$ be a bounded open set with boundary 
$\partial \Omega$. Our point is that we do not need any regularity assumption on
$\Omega$, except that we assume that $\partial \Omega$ has a
finite $(d-1)$-dimensional Hausdorff measure. 
Still we are able to define the Dirichlet-to-Neumann operator on
$L^2(\partial \Omega)$ and to show that  it is selfadjoint and generates a
submarkovian semigroup on $L^2(\Omega)$. Formally, the Dirichlet-to-Neumann
operator $D_0$  is defined as follows. Given $\varphi \in L^2(\Gamma)$, one
solves the Dirichlet problem
\[
 \left\{ \begin{array}{r@{}c@{}l}
\Delta u&{}={}&0 \mbox{ in } \Omega\\
u_{|_{\partial\Omega}}&{}={}& \varphi
\end{array} \right.
\]
and defines $D_0\varphi =\frac{\partial u}{\partial \nu}$.
We will  give a precise definition using weak derivatives.
We consider the space $L^2(\partial \Omega):=L^2(\partial\Omega, \cH^{d-1})$
with the $(d-1)$-dimensional Hausdorff  measure $\cH^{d-1}$. Integrals over
$\partial \Omega$ are always taken with respect to $\cH^{d-1}$, those over
$\Omega$ always with respect to the Lebesgue measure.
Throughout this section we only assume that $\cH^{d-1}(\partial\Omega) < \infty$
and that $\Omega$ is bounded.

\begin{definition} (normal derivative). \label{def:8.1}
Let $u \in H^1(\Omega)$ be such that $\Delta u \in L^2(\Omega)$. We   say that
\[
\frac{\partial u}{\partial \nu} \in L^2(\partial \Omega)
\]
if there exists a $g \in L^2(\partial\Omega)$ such that
\[
\int\limits_\Omega (\Delta u) v + \int\limits_\Omega \nabla u \nabla v =
\int\limits_{\partial\Omega} g v
\]
for all $v \in H^1(\Omega) \cap C(\overline{\Omega})$. This determines $g$ uniquely and we let
$\frac{\partial u}{\partial \nu} := g$.
\end{definition}

Recall that for all $u\in L^1_{\loc}(\Omega)$ the Laplacian $\Delta u$ is defined
in the sense of distributions. If $\Delta u = 0$, then $u \in C^\infty(\Omega)$
by elliptic regularity. Next we define traces of a function $u \in H^1(\Omega)$.

\begin{definition} (traces). \label{def:8.2}
Let $u\in H^1(\Omega)$. We let
\begin{eqnarray*}
\tr(u)
= \{g\in L^2(\Omega)& :& \exists\, (u_n)_{n \in \N} \mbox{ in } H^1(\Omega)\cap C(\overline{\Omega})
\mbox{ such that }  \\*
& & \lim\limits_{n\to \infty} u_n=u \mbox{ in } H^1(\Omega) \mbox{ and }  \\*
& & \lim\limits_{n\to \infty} u_n{}_{|_{\partial\Omega}}=g \mbox{ in }
L^2(\partial \Omega)\} \ .
\end{eqnarray*}
\end{definition}

For arbitrary open sets and $u \in H^1(\Omega)$ the set 
$\tr(u)$ might be empty, or contain more than one element.
However, if $\Omega$ is a
Lipschitz domain, then for each $u\in H^1(\Omega)$ 
the set $\tr(u)$ contains precisely one element, which 
we denote by $u_{|_{\partial\Omega}} \in L^2(\partial\Omega)$.
Now we are in the position to define the Dirichlet-to-Neumann
operator $D_0$.
Its domain is given by
\begin{eqnarray*}
D(D_0) := \{\varphi \in L^2(\partial \Omega) 
&:& \exists\, u \in H^1(\Omega) \mbox{ such that }  \\*
& & \Delta u = 0, \; \varphi \in \tr (u) \mbox{ and }
\frac{\partial u}{\partial \nu} \in L^2(\partial \Omega)
\} 
\end{eqnarray*}
and we define
\[
D_0 \varphi = \frac{\partial u}{\partial \nu}
\]
where $u \in H^1(\Omega)$ is such that 
$\Delta u=0$, $\frac{\partial u}{\partial \nu} \in L^2(\partial \Omega)$
and $\varphi \in \tr(u)$. It is part of our result that
this operator is well-defined.

\begin{theorem} \label{thm:8.3}
The operator $D_0$ is selfadjoint and $-D_0$ generates a submarkovian semigroup
on $L^2(\partial \Omega)$.
\end{theorem}

In the proof we use Theorem \ref{thm:6.4}. Here a non-injective mapping $j$ is
needed.
We also need Maz'ya's inequality. Let $q=\frac{2d}{d-1}$. There exists a
constant
$c_M > 0$ such that
\[
\Big(\int\limits_\Omega |u|^q\Big)^{2/q} \le c_M \Big(\int\limits_\Omega | \nabla u|^2 +
\int\limits_{\partial \Omega} | u|^2\Big)
\]
for all $u \in H^1(\Omega) \cap C(\overline{\Omega})$.
(See \cite[Example 3.6.2/1 and Theorem 3.6.3]{Maz} and \cite[(19)]{AW2}.)

\begin{proof}[Proof of Theorem \ref{thm:8.3}]
We consider  real spaces. Our Hilbert space is $L^2(\partial \Omega)$. Let
$D(a)= H^1(\Omega) \cap C(\overline{\Omega})$, $a(u,v)=\int\limits_\Omega \nabla u
\nabla v$ and define $j \colon D(a) \to L^2(\partial \Omega)$ by
$j(u)=u_{|_{\partial\Omega}} \in L^2(\partial \Omega)$.
Then $a$ is symmetric and $a(u)\ge 0$ for all $u \in D(a)$. Thus the
sectoriality condition before Theorem \ref{thm:6.4} is trivially satisfied.
Denote by $A$ the operator on $L^2(\partial\Omega)$ associated with $(a,j)$.
Let $\varphi,\psi \in L^2(\partial\Omega)$.
Then $\varphi \in D(A)$ and $A\varphi=\psi$ if and only if there exists a sequence
$(u_n)_{n \in \N}$ in $H^1(\Omega) \cap C(\overline{\Omega})$ such that 
$\lim\limits_{n\to \infty} u_n{}_{|_{\partial\Omega}} = \varphi$ in $L^2(\partial\Omega)$,
$\lim\limits_{n\to \infty} a(u_n,v) = \int\limits_{\partial\Omega}\psi v_{|_{\partial\Omega}}$ 
for all $v \in D(a)$ 
and $\lim\limits_{n,m\to \infty}\int\limits_\Omega |\nabla (u_n-u_m)|^2=0$
(here we use Remark~\ref{rem:6.5}). 
Now Maz'ya's inequality implies that $(u_n)_{n\in\N}$ is a Cauchy sequence in
$H^1(\Omega)$.
Thus $\lim\limits_{n\to \infty}u_n=u$ exists in
$H^1(\Omega)$, and so $\varphi \in \tr(u)$. Moreover
$\int\limits_{\partial\Omega}\psi v = \lim\limits_{n\to \infty}
\int\limits_\Omega \nabla u_n \nabla v = \int\limits_\Omega \nabla u \nabla v$
for all $v \in H^1(\Omega) \cap C(\overline{\Omega})$. 
Taking as $v$  test functions, we see that $\Delta u = 0$. Thus
\[
\int\limits_\Omega \nabla u \nabla v + \int\limits_\Omega (\Delta u) v =
\int\limits_{\partial\Omega}\psi v
\]
for all $v \in H^1(\Omega)$. Consequently, $\frac{\partial u}{\partial \nu}  =
\psi$. We have shown that $A\subset D_0$.

Conversely, let $\varphi \in D(D_0),D_0\varphi=\psi$.
Then there exists a $u \in H^1(\Omega)$ such that $\Delta u=0$, $\varphi \in \tr(u)$
and $\frac{\partial u}{\partial v}=\psi$. Since $\varphi \in \tr(u)$ there
exists a sequence $(u_n)_{n \in \N}$ in 
$H^1(\Omega) \cap C(\overline{\Omega})$ such that $u_n\to u$ in
$H^1(\Omega)$ and $u_n{}_{|_{\partial\Omega}} \to \varphi$ in $L^2(\partial
\Omega)$. It follows that $j(u_n)=u_n{}_{|_{\partial\Omega}} \to \varphi$ in
$L^2(\partial\Omega)$, the sequence $(a(u_n))_{n\in\N}$ is bounded and 
\[
a(u_n,v)=\int\limits_\Omega \nabla u_n \nabla v \to \int\limits_\Omega \nabla u
\nabla v =
\int\limits_\Omega \nabla u \nabla v + \int\limits_\Omega (\Delta u) v =
\int\limits_{\partial \Omega} \psi v
\]
for all $v\in H^1(\Omega) \cap C(\overline{\Omega})$. Thus, $\varphi \in D(A)$ and
$A\varphi = \psi$ by the definition of the associated operator.
Since $a$ is symmetric, the operator $A$ is selfadjoint. Now the claim follows from Theorem
\ref{thm:6.4}.

Our criteria easily apply and show that semigroup generated by $-D_0$ is
submarkovian.
\end{proof}

\end{document}